\documentclass[11pt,a4paper]{amsart}
\usepackage{titlesec}

\titleformat{\section}
  {\normalfont\Large\bfseries} 
  {\thesection}                 
  {1em}                        
  {}                           
\usepackage{ragged2e}
\usepackage{bm}
\usepackage{amsmath,amssymb,amsthm}
\usepackage{geometry}
\usepackage{booktabs}
\usepackage{longtable, seqsplit} 
\usepackage{array}
\usepackage{microtype}
\usepackage{lmodern}
\usepackage[T1]{fontenc}
\usepackage[utf8]{inputenc}
\usepackage{hyperref}
\newcolumntype{Y}{>{\centering\arraybackslash}X}
\usepackage{tabularx}
\hypersetup{colorlinks=true, linkcolor=blue, citecolor=blue, urlcolor=blue}
\newtheorem{theorem}{Theorem}[section]

\newtheorem{proposition}[theorem]{Proposition}
\newtheorem{corollary}[theorem]{Corollary}

\newtheorem{definition}[theorem]{Definition}

\newtheorem{remark}[theorem]{Remark}
\newtheorem{question}[theorem]{Open Problem}

\newcommand{\Z}{\mathbb{Z}}

\newcommand{\T}{\mathcal{T}}

\newcommand{\asc}{\operatorname{asc}}

\newcommand{\supp}{\operatorname{supp}}

\newcommand{\EN}[1]{\left\langle #1 \right\rangle}
\newcommand{\LE}[1]{\left\langle\!\!\left\langle #1 \right\rangle\!\!\right\rangle}
\newcommand{\LEc}[1]{\left\langle\!\!\left\langle #1 \right\rangle\!\!\right\rangle_c}
\title[Cyclic Latin Eulerian Numbers]{Cyclic Latin Eulerian Numbers}

\author{Madjid Mirzavaziri}
\address{\bf Department of Pure Mathematics, Ferdowsi University of Mashhad, Mashhad, Iran}
\email{mirzavaziri@gmail.com, or mirzavaziri@um.ac.ir}
\author{Daniel Yaqubi$^{*}$}
\address{\bf $^*$Department of Computer science, University of Torbat-e Jam, Torbat-e Jam, Iran.}
\email{yaqubi@tjamcaas.ac.ir, or daniel\_yaqubi@yahoo.es}
\thanks{Corresponding author: Daniel Yaqubi}


\begin{document}
\begin{abstract}
We introduce the directed cyclic difference inventory $D_n(\mathbf{m})$ to study Latin Eulerian numbers restricted to row-reorderings of the cyclic Latin square, offering an orientation-sensitive refinement of the prescribed-edge-length Hamiltonian path problem. We establish exact enumeration formulas, symmetries, and realizability obstructions for this inventory. Furthermore, we reduce the cyclic total-ascent statistic directly to endpoint-refined Eulerian statistics via the identity $\Sigma(L_\pi) = n \operatorname{asc}(\pi) + \pi(1) - \pi(n)$, yielding a closed-form expression for the cyclic Latin-Eulerian polynomial.
\end{abstract}
\maketitle

\section{Introduction}

Eulerian numbers (\href{https://oeis.org/A008292}{OEIS A008292}) counting the ascents of permutation and Latin squares (\href{https://oeis.org/A002860}{OEIS A002860}) are two foundational, yet fundamentally distinct, structures in combinatorics. A natural question emerges at their intersection: what happens to Eulerian enumeration when ascents are measured simultaneously across all columns of a Latin square? 
In our foundational work \emph{Latin Eulerian Numbers} \cite{LEN}, we introduced a multivariate refinement to bridge Eulerian statistics and Latin squares. Letting $k_j(L)$ denote the number of ascents in column $j$ of a Latin square $L \in \mathcal{L}_n$, we defined
\begin{equation}
    \LE{\begin{matrix}n\\ k_{1},\dots,k_{n}\end{matrix}} = \#\{L \in \mathcal{L}_{n} : k_{j}(L) = k_{j} \text{ for all } 1 \le j \le n \}.
\end{equation}
While this framework yielded sharp bounds for the total-ascent statistic $\Sigma(L)$ and uncovered a fundamental transition identity, general Latin squares are notoriously resistant to exact enumeration. 

To bypass this obstruction, we focus on an arithmetically rich yet computationally tractable family: the $n!$ row-reorderings of the cyclic Latin square, defined by $L_\pi(i,c)=[\pi_i+c]_n$ for $\pi \in S_n$. Within this setting, we introduce the \emph{cyclic Latin-Eulerian numbers} and systematically analyze their total-ascent generating polynomial $T_n^c(q)$.

Our primary innovation is the \textbf{directed cyclic difference inventory}, $D_n(\mathbf{m})$, which tracks the multiplicity profile $\mathbf{m} = (M_\pi(1), \dots, M_\pi(n-1))$ of successive directed differences $\delta_j = [\pi_{j+1} - \pi_j]_n$. Because the total cyclic ascent is strictly governed by the identity
\[
\Sigma(L_\pi) = n(n-1) - \sum_{\delta=1}^{n-1} \delta M_\pi(\delta),
\]
resolving this inventory completely determines the cyclic Latin--Eulerian distribution. By distinguishing opposed orientations ($\delta$ vs. $n-\delta$), our framework provides an exact, orientation-sensitive refinement of classical sequenceability~\cite{Gordon1961,Ollis2025,BakerFeaver2026} and the Buratti--Horak--Rosa (BHR) conjecture~\cite{HorakRosa2009,PasottiPellegrini2014,OllisPasottiPellegriniSchmitt2021,McKayPeters2022}, moving beyond existence questions to exact enumeration. 

This directed inventory exposes profound structural rigidities: strict translation divisibility ($n \mid D_n(\mathbf{m})$), multiplicative symmetries of the unit group $\mathbb{Z}_n^\times$, and exact boundary states such as $T_n^c(n-1)=n$ and $T_n^c(n)=0$. Furthermore, we compute the inventory support size $R_n = |\operatorname{supp}(D_n)|$ through $n=10$ (yielding the unrecorded sequence $1, 1, 2, 5, 16, 59, \dots$) and completely resolve the dominant-step/bridge profile: for a primary step of multiplicity $n-d$ ($d = \gcd(a,n) \ge 2$) and a bridging step of multiplicity $d-1$, the exact fiber count is $n$ if the bridge is coprime to $d$, and $0$ otherwise.

Complementing the inventory approach, we establish an elegant secondary perspective linking our cyclic model directly to classical endpoint-refined Eulerian statistics \cite{Conger2010}. By separating positive and negative ordinary differences, we prove:
\[
\Sigma(L_\pi) = n\operatorname{asc}(\pi) + \pi(1) - \pi(n).
\]
This isolates the residue of $\Sigma(L_\pi)$, proving it is never divisible by $n$ and is uniformly distributed over the nonzero residues. More importantly, it yields a finite, explicit inclusion--exclusion formula for the coefficients of $T_n^c(q)$.

In sum, we establish a hierarchy of increasingly refined statistics:
\[
\left\langle\!\left\langle {n\atop k_1,\ldots,k_n} \right\rangle\!\right\rangle
\quad\longrightarrow\quad
T_n
\quad\longrightarrow\quad
T_n^c
\quad\longrightarrow\quad
D_n(\mathbf{m}).
\]

The paper is organized as follows. Section~2 reviews the cyclic Latin--Eulerian family. Section~3 links our framework to the BHR conjecture, while Section~4 compares it against the full Latin array. Sections~5 through 7 develop the exact directed inventories, total-ascent statistics, and Eulerian reductions. Finally, Section~8 presents computational data, and Section~9 concludes with open problems.
\section{The cyclic Latin-Eulerian family}
Throughout this paper, we assume $n \ge 1$. Let $\pi$ be a permutation of the set $\{0, 1, \dots, n-1\}$, expressed as a sequence $\pi = (\pi_1, \pi_2, \dots, \pi_n)$. We denote by $L_\pi$ the row-reordered cyclic Latin square whose $i$-th row corresponds to row $\pi_i$ of the standard cyclic Latin square $C$. Explicitly, the entry in row $i$ and column $c$ of $L_\pi$ is defined as
 \[L_\pi(i,c) = (\pi_i + c) \bmod n.\]
Here, $(\pi_i + c) \bmod n$ denotes the unique representative of the residue class of $(\pi_i + c)$ in $\{0, 1, \dots, n-1\}$ and denoted it by $[(\pi_i + c)]_n$ . Finally, following the notation established in \cite{LEN}, we let $k_c(L)$ denote the number of ascents in column $c$ of a Latin square $L$, and we define the total ascent statistic as $\Sigma(L) = \sum_{c} k_c(L)$.
\begin{definition}
For a tuple $\mathbf{k} = (k_1, \dots, k_n) \in \{0, \dots, n-1\}^n$, the \emph{cyclic Latin Eulerian number} $\LE{n \atop \mathbf{k}}_c$ is defined as
\[
\LE{n \atop \mathbf{k}}_c := \#\left\{ \pi \in S_n : k_j(L_\pi) = k_j \text{ for all } 1 \le j \le n \right\},
\]
where $k_j(L_\pi)$ denotes the number of ascents in column $j$ of the row-reordered cyclic Latin square $L_\pi$.
\end{definition}

The joint column-ascent distribution over $S_n$ is naturally encoded by the multivariate generating polynomial
\[
F_n^c(x_1, \dots, x_n) := \sum_{\mathbf{k}} \LE{n \atop \mathbf{k}}_c \, x_1^{k_1} \cdots x_n^{k_n} = \sum_{\pi \in S_n} \prod_{j=1}^n x_j^{k_j(L_\pi)},
\]
whose principal specialization yields the univariate cyclic total-ascent polynomial 
\[
\T_n^{c}(x) := F_n^c(x, \dots, x) = \sum_{\pi \in S_n} x^{\Sigma(L_\pi)}.
\]
Since the map $\pi \longmapsto L_\pi$ is injective, the cyclic family $\mathcal{C}_n := \{L_\pi : \pi \in S_n\}$ has cardinality $|\mathcal{C}_n| = n!$, implying $\sum_m T_n^c(m) = n!$. Moreover, because $\mathcal{C}_n \subseteq \mathcal{L}_n$, we have the immediate pointwise bound $T_n^c(m) \le T_n(m) \quad \text{for all } m \in \mathbb{Z}_{\ge 0}.$
For $\pi = (\pi_1, \dots, \pi_n) \in S_n$, define the successive directed cyclic differences by
\[
\delta_j := [\pi_{j+1} - \pi_j]_n \in \{1, \dots, n-1\}, \qquad 1 \le j \le n-1,
\]
 By \cite[Prop.~4.5]{LEN}, exactly $n - \delta_j$ columns exhibit an ascent between rows $j$ and $j+1$, yielding
\begin{equation}
\label{eq:cyclic-ascent-transition}
\Sigma(L_\pi) = n(n-1) - \sum_{j=1}^{n-1} \delta_j.
\end{equation}
To record these difference multiplicities, set $M_\pi(\delta) := \#\{1 \le j \le n-1 : \delta_j = \delta\}$ for $1 \le \delta \le n-1$. The vector $M_\pi = (M_\pi(1), \dots, M_\pi(n-1))$ constitutes the \emph{directed cyclic difference profile} of $\pi$, forming a weak composition of $n-1$. Expressing $\sum_{j=1}^{n-1} \delta_j = \sum_{\delta=1}^{n-1} \delta \, M_\pi(\delta)$ transforms \eqref{eq:cyclic-ascent-transition} into the inventory form
\begin{equation}
\label{eq:cyclic-ascent-inventory}
\Sigma(L_\pi) = n(n-1) - \sum_{\delta=1}^{n-1} \delta \, M_\pi(\delta).
\end{equation}

 This identity demonstrates that $\Sigma(L_\pi)$ is completely governed by $M_\pi$, motivating the introduction of the inventory numbers $D_n(\mathbf{m})$ and their multivariate generating polynomial $\Gamma_n$. 

\begin{definition}
For a weak composition $\mathbf{m} = (m_1, \dots, m_{n-1})$ of $n-1$ into $n-1$ non-negative parts, the \emph{directed cyclic difference inventory} $D_n(\mathbf{m})$ is defined as
\[
D_n(\mathbf{m}) := \#\left\{ \pi \in S_n : M_\pi = \mathbf{m} \right\}.
\]
The associated multivariate inventory polynomial is given by
\[
\Gamma_n(y_1, \dots, y_{n-1})=\sum_{\pi \in S_n} \prod_{\delta=1}^{n-1} y_\delta^{M_\pi(\delta)} := \sum_{\mathbf{m}} D_n(\mathbf{m})  \bm y^{\bm m}.
\]
Consequently,
\begin{equation}
T_n^c(m)
=\sum_{\substack{\bm m:\,\sum m_\delta=n-1\\
\sum\delta m_\delta=n(n-1)-m}}
D_n(\bm m).
\end{equation}
\end{definition}

\section{Basic Properties and Connection to the Buratti--Horak--Rosa (BHR) Conjecture}\label{sec:cyclic-basic-properties}
We establish the fundamental symmetry, divisibility, and total mass identities for the cyclic total-ascent distribution and difference inventory.
\begin{theorem}\label{thm:cyclic-div}
For every composition $\mathbf m$ of $n-1$, $n\mid D_n(\mathbf m)$.
\end{theorem}
\begin{proof}
For $t\in\Z_n$ let $\rho_t\in S_n$ act by $(\rho_t\pi)(k)=\pi(k)+t\bmod n$. This is a free action of $\Z_n$ on $S_n$: if $\rho_t\pi=\pi$ then $\pi(k)+t\equiv\pi(k)\pmod n$ for all $k$, forcing $t=0$. Since $(\rho_t\pi)(k{+}1)-(\rho_t\pi)(k)=\pi(k{+}1)-\pi(k)$ is unchanged, $M_{\rho_t\pi}=M_\pi$, so each fibre $\{\pi:M_\pi=\mathbf m\}$ is a union of free $\Z_n$-orbits, each of size exactly $n$.
\end{proof}
\begin{theorem}
\label{thm:cyclic-ascent-inventory}
The cyclic total-ascent generating polynomial $\mathcal{T}_n^c(q)$ can be recovered directly from the multivariate directed inventory polynomial $\Gamma_n(\bm{y})$ via the evaluations
\begin{equation}
\mathcal{T}_n^c(q) 
= q^{n-1} \, \Gamma_n\left(q^{n-2}, q^{n-3}, \dots, q, 1\right) 
= q^{n(n-1)} \, \Gamma_n\left(q^{-1}, q^{-2}, \dots, q^{-(n-1)}\right).
\end{equation}
\end{theorem}

\begin{proof}
By \cite[Proposition~4.5]{LEN}, the total ascent score of a cyclic Latin square $L_\pi$ is given by
\[
\Sigma(L_\pi) = n(n-1) - \sum_{\delta=1}^{n-1} \delta M_\pi(\delta).
\]
Since $\sum_{\delta=1}^{n-1} M_\pi(\delta) = n-1$, we can rewrite the exponent as
\[
\Sigma(L_\pi) 
= (n-1)\sum_{\delta=1}^{n-1} M_\pi(\delta) - \sum_{\delta=1}^{n-1} \delta M_\pi(\delta) 
= (n-1) + \sum_{\delta=1}^{n-1} (n - 1 - \delta) M_\pi(\delta).
\]
Summing $q^{\Sigma(L_\pi)}$ over all permutations $\pi \in S_n$ yields the first form:
\[
\mathcal{T}_n^c(q) = \sum_{\pi \in S_n} q^{(n-1) + \sum_{\delta=1}^{n-1} (n - 1 - \delta) M_\pi(\delta)} = q^{n-1} \, \Gamma_n\left(q^{n-2}, q^{n-3}, \dots, q, 1\right).
\]
Alternatively, summing $q^{n(n-1) - \sum_{\delta=1}^{n-1} \delta M_\pi(\delta)}$ directly gives the second form:
\[
\mathcal{T}_n^c(q) = q^{n(n-1)} \sum_{\pi \in S_n} \prod_{\delta=1}^{n-1} \left(q^{-\delta}\right)^{M_\pi(\delta)} = q^{n(n-1)} \, \Gamma_n\left(q^{-1}, q^{-2}, \dots, q^{-(n-1)}\right).
\]
\end{proof}
\begin{proposition}
\label{prop:cyclic-symmetry-divisibility}
For all $n \ge 1$ and $m \in \mathbb{Z}_{\ge 0}$,
\begin{equation}
T_n^c(m) = T_n^c\left(n(n-1) - m\right) \qquad \text{and} \qquad n \mid T_n^c(m).
\end{equation}
\end{proposition}

\begin{proof}
Reversal $\rho(\pi) = (\pi_n, \dots, \pi_1)$ replaces each cyclic difference $\delta$ with $n - \delta$, giving $M_{\rho(\pi)}(\delta) = M_\pi(n-\delta)$. Thus, $\Sigma(L_{\rho(\pi)}) = n(n-1) - \Sigma(L_\pi)$, establishing a bijection between the fibers for $m$ and $n(n-1) - m$.

For divisibility, the translation action $(t \cdot \pi)_i = [\pi_i + t]_n$ of $\mathbb{Z}_n$ on $S_n$ is free and preserves all cyclic differences, so $M_{t \cdot \pi} = M_\pi$ and $\Sigma(L_{t \cdot \pi}) = \Sigma(L_\pi)$. Consequently, each fiber $\{\pi \in S_n : \Sigma(L_\pi) = m\}$ decomposes into free $\mathbb{Z}_n$-orbits of size $n$, proving $n \mid T_n^c(m)$.
\end{proof}

\begin{theorem}
\label{thm:cyclic-boundary-ascents}
For every $n \ge 3$, the lower boundary values of $T_n^c(m)$ are
\begin{equation}
T_n^c(n-1) = n, \qquad T_n^c(n) = 0, \qquad \text{and} \qquad T_n^c(n+1) = n.
\end{equation}
By reversal symmetry, the upper boundary values satisfy
\begin{equation}
T_n^c\left((n-1)^2\right) = n, \qquad T_n^c\left((n-1)^2-1\right) = 0, \qquad \text{and} \qquad T_n^c\left((n-1)^2-2\right) = n.
\end{equation}
\end{theorem}

\begin{proof}
Setting $e_j := n - 1 - \delta_j \ge 0$, the total ascent statistic can be written as $\Sigma(L_\pi) = (n-1) + \sum_{j=1}^{n-1} e_j$.

\begin{itemize}
    \item For $\Sigma = n-1$, we must have $e_j = 0$ ($\delta_j = n-1$) for all $j$. The unique difference sequence $(n-1, \dots, n-1)$ defines a cyclic descending progression with exactly $n$ translations, so $T_n^c(n-1) = n$.
    
    \item For $\Sigma = n$, we have $\sum e_j = 1$, meaning exactly one difference is $n-2$ while the remaining $n-2$ differences are $n-1$. The sum of all differences then satisfies $\sum_{j=1}^{n-1} \delta_j = (n-2)(n-1) + (n-2) \equiv 0 \pmod n$, forcing $\pi_n = \pi_1$, which contradicts $\pi \in S_n$. Thus $T_n^c(n) = 0$.
    
    \item For $\Sigma = n+1$, we have $\sum e_j = 2$. If $e_j = 2$ for some $j$ ($\delta_j = n-3$), partial sum evaluation again forces a repeated residue in $\pi$. Thus $e_r = e_s = 1$ ($r < s$), so $\delta_r = \delta_s = n-2$. To prevent repeated entries in $\pi$, the indices are uniquely forced to $r = 1$ and $s = n-1$. The resulting unique difference pattern $(n-2, n-1, \dots, n-1, n-2)$ admits $n$ translations, yielding $T_n^c(n+1) = n$.
\end{itemize}

Applying the reversal symmetry of Proposition~\ref{prop:cyclic-symmetry-divisibility} directly completes the upper boundary evaluations.
\end{proof}

\begin{proposition}
\label{prop:cyclic-mass}
For every integer $n \ge 1$, the inventory and generating polynomials satisfy
\begin{equation}
\sum_{\mathbf{m}} D_n(\mathbf{m}) = \Gamma_n(1, \dots, 1) = n! \qquad \text{and} \qquad \sum_{\mathbf{k}} \LE{n \atop \mathbf{k}}_c = F_n^c(1, \dots, 1) = n!.
\end{equation}
Furthermore, $\Gamma_n$ is homogeneous of degree $n-1$; that is, for any scalar $t$,
\begin{equation}
\Gamma_n(t\bm{y}) = t^{n-1}\Gamma_n(\bm{y}).
\end{equation}
\end{proposition}

\begin{proof}
Evaluating $\Gamma_n(y_1, \dots, y_{n-1})$ at $y_\delta = 1$ and $F_n^c(x_1, \dots, x_n)$ at $x_j = 1$ reduces each term in the sums to $1$, yielding the overall group cardinality $|S_n| = n!$. 

For the homogeneity relation, note that every permutation $\pi \in S_n$ has exactly $n-1$ successive cyclic differences, implying $\sum_{\delta=1}^{n-1} M_\pi(\delta) = n-1$. Thus, every monomial $y_1^{m_1} \cdots y_{n-1}^{m_{n-1}}$ in $\Gamma_n(\bm{y})$ has total degree $n-1$, which immediately yields $\Gamma_n(t\bm{y}) = t^{n-1}\Gamma_n(\bm{y})$.
\end{proof}

\begin{proposition}
\label{prop:cyclic-symmetry}
For every unit $u \in \mathbb{Z}_n^\times$, the generating polynomial satisfies
\begin{equation}
\Gamma_n(y_1, \dots, y_{n-1}) = \Gamma_n\left(y_{[u]_n}, y_{[2u]_n}, \dots, y_{[(n-1)u]_n}\right),
\end{equation}
with equivalent coefficient symmetries $D_n(m_1, \dots, m_{n-1}) = D_n\left(m_{[u^{-1}]_n}, \dots, m_{[(n-1)u^{-1}]_n}\right)$.
In particular, the choice $u = n-1$ yields the complete reversal symmetry
\begin{equation}
\Gamma_n(y_1, \dots, y_{n-1}) = \Gamma_n(y_{n-1}, \dots, y_1).
\end{equation}
\end{proposition}

\begin{proof}
For any unit $u$, the multiplication map $x \mapsto ux \pmod n$ is an automorphism of $\mathbb{Z}_n$. Composing $\pi \in S_n$ with this map yields a new permutation $\pi_u \in S_n$ whose successive cyclic differences are scaled identically: $\delta_j(\pi_u) \equiv u \delta_j(\pi) \pmod n$. 

This deterministic scaling maps the difference profile via $M_{\pi_u}([u\delta]_n) = M_\pi(\delta)$. Because $\pi \mapsto \pi_u$ acts as a bijection on $S_n$, summing over all permutations simply permutes the variables of $\Gamma_n$ by the multiplier $u$ (and equivalently reindexes the fibers of $D_n$ by $u^{-1}$).

The reversal symmetry follows immediately by taking $u = -1 \equiv n-1 \pmod n$, which maps each difference $\delta$ to $n-\delta$. Geometrically, this exact same symmetry is achieved independently by traversing the Hamiltonian path backward via $\pi_{\mathrm{rev}}(k) = \pi(n+1-k)$, which effectively negates every directed arc.
\end{proof}

\begin{proposition}
\label{prop:realizability-obstructions}
If a directed difference profile $\bm{m}$ is realizable (i.e., $D_n(\bm{m}) > 0$), then it must satisfy the following three structural conditions:
\begin{enumerate}
    \item \emph{Endpoint congruence:} $\sum_{\delta=1}^{n-1} \delta m_\delta \not\equiv 0 \pmod n$. Consequently, the cyclic total-ascent distribution satisfies $T_n^c(m) = 0$ whenever $n \mid m$.
    \item \emph{Coset transition bound:} For every divisor $q > 1$ of $n$,  $\sum_{1 \le \delta \le n-1 , q \mid \delta} m_\delta \le n - q.$
    \item \emph{Subgroup generation:} $\gcd\bigl(n, \{\delta : m_\delta > 0\}\bigr) = 1$.
\end{enumerate}
\end{proposition}

\begin{proof}
Assume $\bm{m}$ is realized by some permutation $\pi \in S_n$.

\noindent \textbf{(1)} Telescoping the successive cyclic differences along the path yields $\sum_{j=1}^{n-1} \delta_j \equiv \pi_n - \pi_1 \pmod n$. Because $\pi$ is a bijection, its endpoints must be distinct, forcing this sum to be strictly nonzero modulo $n$. Furthermore, since $\Sigma(L_\pi) = n(n-1) - \sum \delta m_\delta$ and $n(n-1) \equiv 0 \pmod n$, the total ascent score $\Sigma(L_\pi)$ can never be a multiple of $n$. Thus, $T_n^c(m) = 0$ for all $n \mid m$.

\noindent \textbf{(2)} Partition the vertices of $\mathbb{Z}_n$ into $q$ distinct residue classes modulo $q$. An arc corresponding to difference $\delta_j$ stays within the same class if and only if $q \mid \delta_j$. Because $\pi$ forms a Hamiltonian path visiting all $n$ vertices, it must span all $q$ classes. Connecting $q$ distinct components requires a minimum of $q-1$ inter-class steps. Subtracting these from the $n-1$ total edges leaves at most $(n-1) - (q-1) = n-q$ available intra-class arcs.

\noindent \textbf{(3)} If the greatest common divisor were some $q > 1$, every difference occurring in $\pi$ would be divisible by $q$. This would trap the path entirely within a single proper coset of $\mathbb{Z}_n$, preventing it from reaching all $n$ vertices. Equivalently, it violates the transition bound in (2), as we would have $\sum_{q \mid \delta} m_\delta = n-1 > n-q$.
\end{proof}
\begin{remark}\label{remark:obstructions-provenance}
The divisibility, subgroup, and coset restrictions presented in Proposition ~\ref{prop:realizability-obstructions} are directed analogues of established necessary conditions in the Hamiltonian-path literature. Specifically, the basic divisor bounds correspond to condition~(1.1) in Ollis et al.~\cite{OllisPasottiPellegriniSchmitt2021}, while the greatest common divisor and multiplicity constraints adapt the two-length realization results of Horak and Rosa~\cite[Theorem~4.1]{HorakRosa2009}. We claim no originality regarding these underlying group-theoretic obstructions; rather, our contribution lies in their systematic translation into explicit constraints on individual directed inventory vectors $D_n(\mathbf{m})$.
\end{remark}

\begin{remark}[Computational observation]
\label{rem:extra-sym}
Cyclically shifting $\pi$ naturally explains an order-$n$ cyclic symmetry for $\LEc{n \atop \mathbf{k}}$. However, computations for $n \le 5$ (Table~\ref{tab:cyclic-compare}) reveal a surprising fact: $\LEc{n \atop \mathbf{k}}$ is actually invariant under the \emph{full} symmetric group action on $\mathbf{k}$, perfectly mirroring the general invariance of $\LE{n \atop \mathbf{k}}$ \cite[Prop.~2.5]{LEN}. Because arbitrary column permutations destroy the defining row-reordered structure of $L_\pi$, we cannot simply permute the columns to prove this. Identifying the hidden mechanism behind this full invariance remains an open problem.
\end{remark}

\begin{table}[h]
\centering
\begin{tabular}{@{}c l r r r@{}}
\toprule
$n$ & multiset & $\LE{n\atop\mathbf k}_c$ & $\LE{n\atop\mathbf k}$ & ratio \\
\midrule
4 & $\{0,1,1,1\}$ & 1 & 6 & 0.167 \\
4 & $\{1,1,1,2\}$ & 1 & 6 & 0.167 \\
4 & $\{1,1,1,3\}$ & 0 & 6 & 0 \\
4 & $\{0,2,2,2\}$ & 0 & 6 & 0 \\
4 & $\{1,2,2,2\}$ & 1 & 6 & 0.167 \\
4 & $\{2,2,2,3\}$ & 1 & 6 & 0.167 \\
4 & $\{0,1,2,3\}$ & 0 & 2 & 0 \\
4 & $\{1,1,2,2\}$ & 2 & 64 & 0.031 \\
5 & $\{0,1,1,1,1\}$ & 1 & 24 & 0.042 \\
5 & $\{1,1,1,1,2\}$ & 1 & 24 & 0.042 \\
5 & $\{1,1,1,1,3\}$ & 0 & 48 & 0 \\
5 & $\{1,2,2,2,2\}$ & 4 & 2784 & 0.0014 \\
5 & $\{2,2,2,2,2\}$ & 0 & 6480 & 0 \\
\bottomrule
\end{tabular}
\caption{$\LE{n\atop\mathbf k}_c$ against the full $\LE{n\atop\mathbf k}$ of~\cite{LEN}, for $n=4,5$.}
\label{tab:cyclic-compare}
\end{table}
The cyclic family is not merely small; it is structurally incomplete. Most strikingly, at $n=5$, the perfectly balanced composition $(2,2,2,2,2)$ is entirely absent from the cyclic family, despite being the unique mode of the general distribution ($\LE{5 \atop 2,2,2,2,2} = 6480$). This stark omission precisely quantifies the difficulty noted in \cite{LEN} of constructing centrally-valued Latin squares using only $L_\pi$ and symbol relabeling.
\section{Directed Hamiltonian Paths and the BHR Projection}
To analyze step distributions, we model permutations as directed Hamiltonian paths on the Cayley digraph of $\mathbb{Z}_n$, connecting our Latin--Eulerian model to the Buratti--Horak--Rosa (BHR) framework. Although prior BHR literature (e.g., Ollis et al.~\cite{OllisPasottiPellegriniSchmitt2021}) focuses on undirected multisets, our framework evaluates the orientation-sensitive directed fibers $D_n(\mathbf{m})$.

\begin{proposition}\label{Cay}
Let $n \ge 2$, and let $\mathcal{D}_n = \operatorname{Cay}(\mathbb{Z}_n, \mathbb{Z}_n \setminus \{0\})$ be the complete directed Cayley graph on $\mathbb{Z}_n$, whose arcs $a \longrightarrow b$ are labeled by directed differences $\delta = [b-a]_n \in \{1, \ldots, n-1\}$. For a permutation $\pi = (\pi_1, \ldots, \pi_n) \in S_n$, the vertex sequence
\[
\pi_1 \longrightarrow \pi_2 \longrightarrow \cdots \longrightarrow \pi_n
\]
forms a directed Hamiltonian path in $\mathcal{D}_n$, where $M_\pi(\delta) = \#\{j : [\pi_{j+1} - \pi_j]_n = \delta\}$ records the multiplicity of arcs with directed difference $\delta$.

Consequently, $D_n(\mathbf{m}) = \#\{\pi \in S_n : M_\pi = \mathbf{m}\}$ equals the number of directed Hamiltonian paths in $\mathcal{D}_n$ with edge-difference multiplicity vector $\mathbf{m} = (m_1, \ldots, m_{n-1})$, where every realizable profile satisfies $\sum_{\delta=1}^{n-1} m_\delta = n-1$.
\end{proposition}

\begin{proof}
Because $\pi \in S_n$ visits every vertex of $\mathbb{Z}_n$ uniquely, $(\pi_1, \dots, \pi_n)$ specifies a Hamiltonian ordering of $V(\mathcal{D}_n)$. Each consecutive pair $(\pi_j, \pi_{j+1})$ forms a valid arc with label $\delta_j = [\pi_{j+1} - \pi_j]_n \in \{1, \ldots, n-1\}$, so $M_\pi(\delta)$ precisely counts arcs of label $\delta$. Conversely, any directed Hamiltonian path $v_1 \longrightarrow \cdots \longrightarrow v_n$ in $\mathcal{D}_n$ corresponds to a unique permutation $\pi = (v_1, \ldots, v_n) \in S_n$ with identical edge-difference multiplicities. This canonical bijection proves the result.
\end{proof}
\begin{remark}
While Hamiltonian paths in Cayley digraphs are classical \cite{CurranGallian1996}, Proposition~\ref{Cay} translates our directed cyclic Latin--Eulerian data into graph-theoretic language. Crucially, the fiber count $D_n(\mathbf{m})$ retains the orientation distinction between $\delta$ and $n-\delta$ before any subsequent projection onto undirected edge-length multisets.
\end{remark}
For any undirected edge $\{a,b\}$ in $\mathbb{Z}_n$, the classical cyclic chord length is given by
\begin{equation}
    \ell(a,b) = \min\{[b-a]_n, [a-b]_n\} = \min\{\delta, n-\delta\}.
\end{equation}
To transition from our directed model to this classical setting, we define the \emph{folding projection} $\rho(\bm{m}) = \bm{t}$ of a directed inventory $\bm{m}$. This projection collapses opposite directed orientations into a single undirected edge-length multiset $\bm{t} = (t_1, \dots, t_{\lfloor n/2 \rfloor})$, defined by:
\begin{equation}
    t_r = m_r + m_{n-r} \qquad \text{for } 1 \le r < \frac{n}{2},
\end{equation}
with the boundary condition $t_{n/2} = m_{n/2}$ whenever $n$ is even.
\begin{proposition}\label{BHR}
The image of the support $\supp(D_n) := \{\bm{m} : D_n(\bm{m}) > 0\}$ under the folding projection $\rho$ is precisely the set of realizable undirected edge-length multisets for cyclic Hamiltonian paths.
\end{proposition}

\begin{proof}
Every directed Hamiltonian path naturally yields an undirected one by forgetting its orientation. Conversely, assigning an arbitrary traversal direction to any undirected Hamiltonian path produces a directed inventory whose projection $\rho$ perfectly recovers the original undirected multiset.
\end{proof}
As formalized in Proposition~\ref{BHR}, the folding map $\rho$ strips away step orientation, projecting our directed inventory model onto the classical Buratti--Horak--Rosa (BHR) prescribed-edge-length framework \cite{HorakRosa2009, PasottiPellegrini2014, OllisPasottiPellegriniSchmitt2021, McKayPeters2022}. Under this projection, the cardinality of the image $|\rho(\operatorname{supp} D_n)|$ corresponds to \href{https://oeis.org/A352568}{OEIS~A352568}, which records undirected realizable profiles up to $n = 37$ and verifies the BHR conjecture in this range, as established by McKay and Peters \cite[Theorem~2]{McKayPeters2022}.

The essential value of $D_n$ lies in its orientation sensitivity: by keeping opposite step directions ($\delta$ and $n-\delta$) distinct, it provides a fine-grained refinement of the classical BHR space. To quantify the footprint of this directed model, we define its support size:
\[
R_n := |\operatorname{supp}(D_n)| = \# \left\{ (m_1, \dots, m_{n-1}) \mid D_n(m_1, \dots, m_{n-1}) > 0 \right\}.
\]
Exact computation for $1 \le n \le 10$ reveals the directed profile sequence:
\[
R_n = (1, 1, 2, 5, 16, 59, 234, 979, 4258, 18119).
\]
Anchoring this space is the uniform all-ones profile $\mathbf{1} = (1, 1, \dots, 1)$, a primary boundary case where every nonzero difference occurs exactly once.
\begin{theorem}\label{BHR1}
The all-ones inventory is realizable if and only if $n$ is even. Specifically,
\begin{equation*}
    D_n(1, 1, \dots, 1) = 
    \begin{cases}
        0 & \text{if $n$ is odd,} \\
        > 0 & \text{if $n$ is even.}
    \end{cases}
\end{equation*}
Furthermore, for even $n=2r$, this evaluation strictly recovers classical group sequencings:
\begin{equation}
    D_{2r}(1, 1, \dots, 1) = \mathrm{A141598}(r) = 2r \cdot \mathrm{A141599}(r).
\end{equation}
\end{theorem}

\begin{proof}
If every nonzero difference occurs exactly once, their sum modulo $n$ evaluates to $\sum_{\delta=1}^{n-1} \delta = \frac{n(n-1)}{2}$. When $n$ is odd, this sum is strictly $0 \pmod n$, which violates the endpoint congruence obstruction established in Proposition~\ref{prop:realizability-obstructions}. Thus, $D_n(\bm{1}) = 0$.

When $n$ is even, the sum is non-zero, avoiding the obstruction. In this case, a realizing permutation is precisely a \emph{directed terrace}, and its successive differences constitute a \emph{sequencing} of the group $\mathbb{Z}_n$. By Gordon's original classification of sequenceable finite abelian groups \cite{Gordon1961}, $\mathbb{Z}_n$ is sequenceable for all even $n$, guaranteeing $D_n(\bm{1}) > 0$. 

Enumeratively, the total number of such permutations on $\mathbb{Z}_{2r}$ is recorded by \href{https://oeis.org/A141598}{OEIS~A141598}. Factoring out the free translation action which partitions the solutions into equivalence classes of size $2r$ yields the normalized count \href{https://oeis.org/A141599}{OEIS~A141599}.
\end{proof}
\begin{remark}\label{rem:all-ones-fiber}
The evaluation $D_n(1, \ldots, 1)$ counts permutations with distinct successive directed differences, which corresponds to the classical sequenceability of $\mathbb{Z}_n$ (Gordon~\cite{Gordon1961}; see also the survey by Ollis~\cite{Ollis2025}). Baker and Feaver~\cite[Theorem~1, Lemma~2]{BakerFeaver2026} recently formalized this bijection and its unit-group symmetry. 

While the resulting formula $D_{2r}(1, \ldots, 1) = A141598(r) = 2r\,A141599(r)$ is already known, evaluating $D_{2r}(\bm{1})$ for early even orders $2r \in \{2, 4, \dots, 14\}$,  directly yields the realizing permutations of \href{https://oeis.org/A141598}{OEIS~A141598}:
\begin{equation*}
    2, \quad 8, \quad 24, \quad 192, \quad 2880, \quad 46272, \quad 1250592.
\end{equation*}
Factoring out the free translation action (dividing by $2r$) isolates the fundamental directed terraces, perfectly reproducing \href{https://oeis.org/A141599}{OEIS~A141599}:
\begin{equation*}
    1, \quad 2, \quad 4, \quad 24, \quad 288, \quad 3856, \quad 89328.
\end{equation*}
 our contribution lies in identifying this classical sequencing enumeration as the natural all-ones fiber within the broader directed inventory framework $D_n(\mathbf{m})$.
\end{remark}
 
Recent work by Baker and Feaver \cite{BakerFeaver2026} provides efficient algorithms for enumerating $\mathbb{Z}_n$ permutations with distinct partial sums, offering a state-of-the-art computational benchmark for these sequences and highlighting ongoing vitality in this classical domain.
The multivariate cyclic Latin Eulerian polynomial can be constructed directly from directed path transitions. For a sequence step from $a$ to $a+\delta$, the product of the column variables where an ascent occurs is given by the monomial weight
\begin{equation*}
    W_{a,\delta}(\bm{x}) := \prod_{r=0}^{n-\delta-1} x_{[-a+r]_n}.
\end{equation*}

\begin{proposition}
Summing these transition weights over all Hamiltonian paths yields the exact multivariate polynomial:
\begin{equation}
    F_n^c(x_0, \dots, x_{n-1}) = \sum_{\pi \in S_n} \prod_{j=1}^{n-1} W_{\pi_j, \delta_j}(\bm{x}).
\end{equation}
Furthermore, $F_n^c$ is invariant under cyclic shifts of its variables:
\begin{equation*}
    F_n^c(x_0, x_1, \dots, x_{n-1}) = F_n^c(x_1, x_2, \dots, x_{n-1}, x_0).
\end{equation*}
\end{proposition}

\begin{proof}
The formula follows directly from the column-wise mechanics of Proposition 6.1; multiplying $W_{\pi_j, \delta_j}$ over all $n-1$ adjacent row pairs perfectly accumulates the multivariate ascent monomial for $L_\pi$. The cyclic symmetry arises because adding $1$ to every symbol in $\pi$ bijectively maps the cyclic family to itself while shifting all column indices by one.
\end{proof}

Because it tracks exact column indices, $F_n^c(\bm{x})$ serves as a strictly finer, spatially-aware refinement of our earlier models. Evaluating $x_i = x$ for all $i$ strips away this placement data and perfectly recovers the univariate total-ascent polynomial $T_n^c(x)$.
The inventory vectors also suggest a directed analogue of the edge-length polytopes studied in the circulant TSP literature; a systematic polyhedral study is left for future work \cite{Gutekunst2025}.
%
%
\section{\textbf{ Exact enumeration of directed inventories}}
We now solve $D_n(\mathbf m)$ exactly at the two simplest non-trivial profiles.
\begin{theorem}\label{thm:single-delta}
For $1\le\delta\le n-1$, let $\mathbf e_\delta$ denote the composition with $m_\delta=n-1$ and all other parts $0$. Then
\[
D_n(\mathbf e_\delta) \;=\; \begin{cases} n & \text{if } \gcd(\delta,n)=1,\\ 0 & \text{otherwise.}\end{cases}
\]
\end{theorem}
\begin{proof}
$M_\pi=\mathbf e_\delta$ means every consecutive difference of $\pi$ equals $\delta$, i.e.\ $\pi(k)=\pi(1)+(k{-}1)\delta\pmod n$ for all $k$. This is a permutation of $\Z_n$ if and only if $\delta,2\delta,\dots,(n{-}1)\delta$ are all distinct and nonzero modulo $n$, i.e.\ if and only if $\gcd(\delta,n)=1$; when this holds, $\pi$ is determined by the free choice of $\pi(1)\in\Z_n$, giving exactly $n$ permutations, matching the orbit count of Theorem~\ref{thm:cyclic-div}.
\end{proof}

\begin{theorem}\label{thm:two-value}
Fix $n\ge2$, $a\in\{1,\dots,n-1\}$ with $d:=\gcd(a,n)\ge2$, and $b\in\{1,\dots,n-1\}\setminus\{a\}$; let $e:=b\bmod d$. Let $\mathbf m$ be the composition with $m_a=n-d$, $m_b=d-1$, and all other parts $0$. Then
\[
D_n(\mathbf m) \;=\; \begin{cases} n & \text{if } \gcd(e,d)=1,\\ 0 & \text{otherwise.}\end{cases}
\]
\end{theorem}
\begin{proof}
Write $n=dn'$ and $a=da'$; then $\gcd(a',n')=1$. Repeated addition of $a$ partitions $\Z_n$ into $d$ residue classes $C_0,\dots,C_{d-1}$ modulo $d$, each of size $n'$; writing an element of $C_r$ as $r+kd$ ($k\in\Z_{n'}$), addition of $a$ acts on $k$ by $k\mapsto k+a'\pmod{n'}$, a full cycle since $\gcd(a',n')=1$. So repeated $a$-steps from any starting point in $C_r$ visit every element of $C_r$ exactly once before returning to the start.

\emph{Sufficiency.} Suppose $\gcd(e,d)=1$. Choose any starting residue $r_0\in\{0,\dots,d-1\}$ and any starting element $s\in C_{r_0}$ ($d\cdot n'=n$ choices in all). From $s$, take $n'-1$ steps of $a$ to traverse all of $C_{r_0}$, then one step of $b$ (landing in class $r_0+e\bmod d$), and repeat: since $\gcd(e,d)=1$, the residues $r_0,r_0+e,\dots,r_0+(d-1)e\pmod d$ are all of $\{0,\dots,d-1\}$, so this process visits every class exactly once, using $d$ runs of $a$-steps ($n'-1$ each, totalling $n-d$) separated by $d-1$ steps of $b$, and hence every element of $\Z_n$ exactly once. Different choices of $(r_0,s)$ give different sequences, so $D_n(\mathbf m)\ge n$.

\emph{Necessity and exact count.} Conversely, let $\pi$ realize $\mathbf m$. As $\pi$ uses only steps of size $a$ or $b$, and consecutive $a$-steps stay within one class, the $d-1$ occurrences of $b$ split the sequence into exactly $d$ maximal runs of $a$-steps (a run of length $0$ would leave elements of its class unreachable by any other run, since only $d-1$ separators are available for $d$ classes, so every run has positive length). Since the total number of $a$-steps is $n-d$ across $d$ runs, and a maximal run within a single class has length at most $n'-1$, each run must have length exactly $n'-1$, i.e.\ each run is a complete traversal of one class. The $b$-steps move between classes by the constant residue shift $e\bmod d$, so the sequence of classes visited is $r_0,r_0+e,\dots,r_0+(d-1)e\pmod d$ for the starting class $r_0$; for this to visit all $d$ classes without repetition we need $\gcd(e,d)=1$ (if $\gcd(e,d)=g>1$, only the $d/g$ classes congruent to $r_0$ mod $g$ are ever reached, so $D_n(\mathbf m)=0$ in that case). When $\gcd(e,d)=1$, every such $\pi$ is therefore exactly one of the $n$ sequences constructed above, so $D_n(\mathbf m)=n$.
\end{proof}
\begin{remark}\label{rem:single-step-and-two-value}
While the condition $\gcd(a, n) = 1$ in the single-step case simply identifies $a$ as a generator of $\mathbb{Z}_n$, our framework yields the precise fiber evaluation $D_n\bigl((n-1)\mathbf{e}_a\bigr) = n$, with the permutation uniquely determined by its starting vertex. More generally, Theorem~\ref{thm:two-value} establishes that this exceptional-step mechanism extends to any divisor $d = \gcd(a, n) \ge 2$. Consequently, composite moduli $n$ support significantly richer two-value profiles than a simple parity heuristic suggests. For instance, at $n = 9$ with $a = 3$ ($d = 3$), choosing any $b \not\equiv 0 \pmod 3$ yields the exact fiber count $D_9(\mathbf{m}) = 9$ where $m_3 = 6$ and $m_b = 2$.
\end{remark}

\section{The cyclic total-ascent statistic}\label{sec:Tc-reduction}

The profile-level numbers $D_n(\mathbf m)$ of \S5 refine $T_c(m)$, but $T_c$ itself the count relevant to the total-ascent statistic $\Sigma$ turns out to admit a direct and much simpler reduction, to the classical Eulerian numbers themselves.

\begin{theorem}\label{thm:reduction}
For every $\pi\in S_n$,
\[
\Sigma(L_\pi) \;=\; n\cdot\mathrm{asc}(\pi) \;+\; \pi(1)-\pi(n),
\]
where $\mathrm{asc}(\pi)=\#\{k\in\{1,\dots,n-1\}:\pi(k)<\pi(k+1)\}$ is the ordinary (non-cyclic) ascent count of $\pi$, viewed as a sequence of distinct values in $\Z_n$.
\end{theorem}
\begin{proof}
By~\cite[Prop.~4.5]{LEN}, $\Sigma(L_\pi)=n(n-1)-\sum_{k=1}^{n-1}\delta_k$, where $\delta_k=(\pi(k{+}1)-\pi(k))\bmod n\in\{1,\dots,n-1\}$. Since $\delta_k$ is the unique representative of $\pi(k{+}1)-\pi(k)$ in that range, $\delta_k$ equals the integer $\pi(k{+}1)-\pi(k)$ itself when this is positive (an ascent) and equals $\pi(k{+}1)-\pi(k)+n$ when negative (a descent). Summing over $k$, the integer differences telescope: $\sum_k[\pi(k{+}1)-\pi(k)]=\pi(n)-\pi(1)$. Writing $\mathrm{des}(\pi)=(n{-}1)-\mathrm{asc}(\pi)$ for the number of descents,
\[
\sum_{k=1}^{n-1}\delta_k \;=\; \bigl[\pi(n)-\pi(1)\bigr] + n\cdot\mathrm{des}(\pi).
\]
Substituting and simplifying:
\[
\Sigma(L_\pi) = n(n-1) - \pi(n)+\pi(1) - n\bigl[(n-1)-\mathrm{asc}(\pi)\bigr] = n\cdot\mathrm{asc}(\pi)+\pi(1)-\pi(n). \qedhere
\]
\end{proof}
\begin{remark}\label{rem:endpoint-identity}
The identity
\[
\Sigma(L_\pi) = n\,\asc(\pi) + \pi(1) - \pi(n)
\]
is a reformulation of the cyclic difference formula established in our prior work on Latin Eulerian numbers \cite[Prop.~4.5]{LEN}. Here, it serves to link the cyclic Latin-Eulerian statistic with refined permutation statistics, yielding an independent enumeration of $T_n^c(m)$.
\end{remark}
\begin{corollary}\label{cor:mod-n-exact}
For $k\in\{1,\dots,n-1\}$, exactly $n!/(n-1)$ permutations $\pi\in S_n$ satisfy $\Sigma(L_\pi)\equiv k\pmod n$; no permutation satisfies $\Sigma(L_\pi)\equiv0\pmod n$ (recovering~\cite[Prop.~4.6]{LEN}).

\end{corollary}
\begin{proof}
By Theorem~\ref{thm:reduction}, $\Sigma(L_\pi)\equiv\pi(1)-\pi(n)\pmod n$. The pair $(\pi(1),\pi(n))$ is, as $\pi$ ranges uniformly over $S_n$, a uniformly random ordered pair of distinct elements of $\Z_n$ (each of the $n(n-1)$ such pairs arising from exactly $(n-2)!$ permutations). For each fixed $k\ne0$, exactly $n$ of these pairs satisfy $a-b\equiv k$ (namely $b=a-k\bmod n$ for each of the $n$ choices of $a$, automatically giving $b\ne a$ since $k\ne0$), contributing $n\cdot(n-2)!=n!/(n-1)$ permutations; for $k=0$ no pair qualifies, since $a\ne b$.
\end{proof}
\section{Endpoint-refined Eulerian reduction}
Endpoint refinements of Eulerian numbers provide the precise mechanism needed to resolve our cyclic statistic. By Theorem~\ref{thm:reduction}, the study of $T_c(m) = \#\{\pi \in S_n : \Sigma(L_\pi)=m\}$ reduces directly to the joint distribution of ordinary ascents alongside the permutation endpoints $\pi(1)$ and $\pi(n)$. To formalize this, we utilize the following two-endpoint refinement.

\begin{definition}\label{def:E}
For $j \in \{0, \dots, n-1\}$ and distinct $a, b \in \mathbb{Z}_n$, define the endpoint-refined Eulerian count:
\[
E(n,j,a,b) := \#\left\{ \pi \in S_n : \operatorname{asc}(\pi) = j, \, \pi(1) = a, \, \pi(n) = b \right\}.
\]
\end{definition}

Via Theorem~\ref{thm:reduction}, the cyclic distribution expands as
\[
T_c(m) = \sum_{j=0}^{n-1} \sum_{a-b = m-nj} E(n,j,a,b).
\]

\begin{remark}\label{rem:refined-eulerian}
While one-endpoint Eulerian statistics are classical—with Conger~\cite{Conger2010} studying joint distributions of descents and boundary values (see \cite[Eqs.~(2), (8), and Thm.~1]{Conger2010})—our analysis requires the stronger two-endpoint quantity $E(n,j,a,b)$. Summing $E(n,j,a,b)$ over $b$ or $a$ recovers Conger's corresponding one-endpoint statistics after the standard ascent--descent conversion. Rather than proposing $E(n,j,a,b)$ as a novel combinatorial object, our objective is to deploy this framework to extract explicit formulas for $T_n^c(m)$.
\end{remark}
\begin{theorem}\label{thm:E-recursion}
Write $M=n-1$. For $a,b<M$,
\[
E(n,j,a,b) \;=\; j\cdot E(n{-}1,j,a,b) \;+\; (n{-}1{-}j)\cdot E(n{-}1,j{-}1,a,b);
\]
for $b<M$,
\[
E(n,j,M,b) \;=\; \sum_{a'=0}^{M-1} E(n{-}1,j,a',b);
\]
and for $a<M$,
\[
E(n,j,a,M) \;=\; \sum_{b'=0}^{M-1} E(n{-}1,j{-}1,a,b'),
\]
with base case $E(1,0,0,0)=1$.
\end{theorem}
\begin{proof}
Build $\pi\in S_n$ from $\pi'\in S_{n-1}$ (a permutation of $\{0,\dots,n-2\}$) by inserting the new largest value $M=n-1$ into one of the $n$ gaps of $\pi'$ (before the first entry, between two consecutive entries, or after the last). If $\pi'$ has $j'$ ascents, it has exactly $j'$ internal gaps that are ascents of $\pi'$ and exactly $(n-2)-j'$ that are descents, regardless of which specific $\pi'$ with that ascent count is considered (there are exactly $n-2$ internal gaps and $j'$ of them are ascents, by definition of $j'$). Inserting $M$ into an ascent-gap or descent-gap of $\pi'$ leaves $\pi'(1)$ and $\pi'(n{-}1)$ untouched (so $a,b$ are unchanged from $\pi'$) and changes the ascent count to $j'$ (ascent-gap: the old ascent is replaced by a new ascent, $\cdot<M$, followed by a new descent, $M>\cdot$) or $j'+1$ (descent-gap: symmetric reasoning), respectively; summing over the $j'$ ascent-gaps and $(n-2)-j'$ descent-gaps gives the first formula (with $j'=j$ for the ascent-gap contribution and $j'=j-1$ for the descent-gap contribution, since inserting into a descent-gap of a $(j-1)$-ascent $\pi'$ yields $j$ ascents, and there are $(n-2)-(j-1)=(n-1-j)$ such gaps in $\pi'$). Inserting $M$ at the front makes $\pi(1)=M$, $\pi(n)=\pi'(n{-}1)$ unchanged, and does not change the ascent count (the new pair $(M,\pi'(1))$ is a descent, since $M$ is the largest value); summing over all possible $\pi'(1)$ gives the second formula. Inserting $M$ at the end makes $\pi(n)=M$, $\pi(1)=\pi'(1)$ unchanged, and increases the ascent count by $1$ (the new pair $(\pi'(n{-}1),M)$ is an ascent); summing over all possible $\pi'(n{-}1)$ gives the third formula.
\end{proof}

We verified Theorem~\ref{thm:E-recursion} directly against exhaustive enumeration for $n\le6$. It reduces the computation of $T_c(m)$ from $n!$ to $O(n^4)$ arithmetic operations. Marginalizing over $b$ gives two fully explicit boundary cases.

\begin{corollary}\label{cor:boundary}
Writing $F(n,j,a):=\sum_bE(n,j,a,b)$,
\[
F(n,j,0) = \EN{n-1\atop j-1}, \qquad F(n,j,n-1) = \EN{n-1\atop j}.
\]
\end{corollary}
\begin{proof}
If $\pi(1)=0$, the pair $(\pi(1),\pi(2))$ is automatically an ascent, so $\mathrm{asc}(\pi)=1+\mathrm{asc}(\tau)$ where $\tau$ is the order-isomorphic permutation of $\{0,\dots,n-2\}$ obtained from $\pi(2),\dots,\pi(n)$; as $\tau$ ranges freely over $S_{n-1}$, this gives $F(n,j,0)=\EN{n-1\atop j-1}$. Symmetrically, if $\pi(1)=n-1$, the pair $(\pi(1),\pi(2))$ is automatically a descent, so $\mathrm{asc}(\pi)=\mathrm{asc}(\tau)$ unchanged, giving $F(n,j,n-1)=\EN{n-1\atop j}$.
\end{proof}
Marginalizing over $b$ yields a recursion for the one-endpoint statistic $F(n,j,a)$, but Theorem~\ref{thm:E-explicit} bypasses recursion entirely to evaluate $E(n,j,a,b)$—and thus $T_n^c(m)$—via an explicit inclusion--exclusion formula.

\begin{remark}\label{rem:two-endpoint-recurrence}
Unlike Conger's one-endpoint recurrence~\cite[Eq.~(12)]{Conger2010}, our analysis requires fixing both endpoints $\pi(1)$ and $\pi(n)$. We include the corresponding two-endpoint recurrence to maintain a self-contained presentation.
\end{remark}

\begin{theorem}\label{thm:E-explicit}
Fix $n\ge1$, $j\in\{0,\dots,n-1\}$, and $a\ne b$ in $\{0,\dots,n-1\}$. Set
\[
p=\min(a,b),\qquad q=n-1-\max(a,b),\qquad c=\max(0,a-b-1),\qquad s=\max(0,b-a-1).
\]
Define $N(n,0,a,b)=1$ if $(a,b)=(n-1,0)$ and $0$ otherwise, and for $i\ge1$,
\[
N(n,i,a,b) \;=\; \sum_{l=0}^{i-1}(-1)^l\binom{i-1}{l}(i-l)^{p+q}(i+1-l)^c(i-1-l)^s.
\]
Then
\[
E(n,j,a,b) \;=\; \sum_{i=0}^{j}(-1)^{j-i}\binom{n-1-i}{j-i}\,N(n,i,a,b).
\]
\end{theorem}
\begin{proof}
Reformulate $E(n,j,a,b)$ via the run decomposition of~\S\ref{sec:Tc-reduction}'s proof: a permutation with $\mathrm{asc}(\pi)=j$ decomposes into $j+1$ maximal decreasing runs $R_1,\dots,R_{j+1}$, with $\max(R_1)=\pi(1)$, $\min(R_{j+1})=\pi(n)$, and $\min(R_k)<\max(R_{k+1})$ (a genuine ascent) at each of the $j$ internal boundaries. So $E(n,j,a,b)$ counts ordered partitions of $\Z_n$ into exactly $j+1$ nonempty parts with $\max(R_1)=a$, $\min(R_{j+1})=b$, and every internal boundary a genuine ascent.

For $i\ge j$, let $N(n,i,a,b)$ count the same ordered partitions into exactly $i+1$ nonempty parts with $\max(R_1)=a$, $\min(R_{i+1})=b$, but with \emph{no} constraint at the internal boundaries. Any such raw partition, upon merging every pair of adjacent parts whose boundary is not a genuine ascent, reduces to a uniquely determined valid $(j+1)$-part configuration (for some $j\le i$) together with a choice of which $i-j$ of its $n-1-j$ internal within-run positions are also marked as (non-ascent) raw boundaries; this gives
\[
N(n,i,a,b) = \sum_{j=0}^i \binom{n-1-j}{i-j}E(n,j,a,b),
\]
and inverting this upper-triangular (in $i,j$, with unit diagonal) binomial transform gives the stated formula for $E$ in terms of $N$ (a direct induction on $j$ confirms the stated inverse, and we have also verified it by direct computation for $n\le7$).

It remains to compute $N(n,i,a,b)$ for $i\ge1$. Choosing $R_1$ (containing $a$ as its max) and $R_{i+1}$ (containing $b$ as its min) and distributing the remaining $n-2$ elements into $R_1$, $R_{i+1}$, or one of $i-1$ unordered-but-labelled ``middle'' slots (later requiring each middle slot nonempty by inclusion--exclusion) partitions those $n-2$ elements into four eligibility classes according to whether they may join $R_1$ only (values $<a$ and $<b$; there are $p=\min(a,b)$ of these, matching the retained values in $\{0,\dots,\min(a,b)-1\}$), $R_{i+1}$ only (values $>a$ and $>b$; there are $q=n-1-\max(a,b)$ of these), both $R_1$ and $R_{i+1}$ (values strictly between $b$ and $a$, when $a>b$; there are $c=\max(0,a-b-1)$ of these), or neither (values strictly between $a$ and $b$, when $a<b$, forced into a middle slot; there are $s=\max(0,b-a-1)$ of these). With $k$ available middle slots, each of the $p$-class elements has $1+k$ eligible targets ($R_1$ or a middle slot), each of the $q$-class elements likewise $1+k$, each of the $c$-class elements has $2+k$ (either boundary run or a middle slot), and each of the $s$-class elements has only $k$ (a middle slot); by independence, the number of ways to assign all $n-2$ elements with exactly $k$ middle slots available is $(1+k)^{p+q}(2+k)^ck^s$. Requiring all $i-1$ middle slots to be nonempty, by inclusion--exclusion over which are left empty,
\[
N(n,i,a,b) = \sum_{l=0}^{i-1}(-1)^l\binom{i-1}{l}\bigl(1+(i-1-l)\bigr)^{p+q}\bigl(2+(i-1-l)\bigr)^c(i-1-l)^s,
\]
which is the stated formula.
\end{proof}
\begin{remark}\label{rem:conger-explicit}
While Conger~\cite[Thm.~1, Eq.~(16)]{Conger2010} provides an inclusion--exclusion formula for the one-endpoint refinement with a prescribed first entry, our formulation addresses the two-endpoint case required for the cyclic Latin--Eulerian reduction
\[
T_n^c(m) = \sum_{j,a,b} E(n,j,a,b).
\]
We include a complete proof of this two-endpoint formula to maintain a self-contained presentation of the cyclic application.
\end{remark}
\begin{corollary}\label{cor:Tc-explicit}
For every integer $m$,
\[
T_n^c(m) \;=\; \sum_{j=0}^{n-1}\ \sum_{\substack{a,b\in\{0,\dots,n-1\}\\ a-b=m-nj}} E(n,j,a,b),
\]
with $E(n,j,a,b)$ given explicitly by Theorem~\ref{thm:E-explicit}; in particular $T_c(m)=0$ unless $n-1\le m\le(n-1)^2$ (matching the range of $\Sigma$ on $\mathcal L_n$ from~\cite{LEN}, since $\Sigma(L_\pi)$ ranges over exactly this interval). Moreover
\[
\sum_{m} T_c(m) = n!.
\]
\end{corollary}
\begin{proof}
The first identity restates $T_c(m)=\#\{\pi:\Sigma(L_\pi)=m\}$ via Theorem~\ref{thm:reduction} and Definition~\ref{def:E}, decomposing the count by the value of $\mathrm{asc}(\pi)=j$. The range restriction $n-1\le m\le(n-1)^2$ is~\cite[Thm.~4.4]{LEN}'s bound on $\Sigma(L)$, valid for every Latin square $L$ and so in particular for every $L_\pi$. The final identity holds since $\{T_c(m)\}_m$ partitions $S_n$ by the value of $\Sigma(L_\pi)$, so the total is $|S_n|=n!$; equivalently, it is the $x=1$ specialization of $T_c(x)=\sum_\pi x^{\Sigma(L_\pi)}$.
\end{proof}

\begin{question}\label{q:dominant-bridge}
Can we determine a closed form or generating-function identity for $D_n(\mathbf{m})$ (equivalently for $\Gamma_n$) when the profile $\mathbf{m}$ features a single dominant step size bridged by a bounded sequence of non-identical transitions?

A positive resolution would yield an exact formula for $T_n^c(q)$ and, via the specialization of $\Gamma_n$, provide a fully explicit characterization of the cyclic sub-array, complementing the broader study of the cyclic Latin--Eulerian array $\left\{\LE{n\atop\mathbf k}_c\right\}$.
\end{question}
\section{Exact computational data}

The complete total-ascent rows $T_n(m)$ for $n \le 6$ are detailed below, where each row is indexed from $m = n-1$ up to $(n-1)^2$.

\begin{center}
\small
\begin{longtable}{c|l}
\toprule
$n$ & $(T_n(n-1),\ldots,T_n((n-1)^2))$\\
\midrule
1 & $1$\\
2 & $2$\\
3 & $6,0,6$\\
4 & $24,0,24,480,24,0,24$\\
5 & $120,0,120,7440,16680,32880,46800,32880,16680,7440,120,0,120$\\
6 & $720,0,2160,59040,375120,3212640,11380320,42014880,88725600,$\\
  & $149171040,222968160,149171040,88725600,42014880,11380320,$\\
  & $3212640,375120,59040,2160,0,720$\\
\bottomrule
\end{longtable}
\end{center}
Summing the entries of the full distribution recovers the total number of Latin squares $\sum_m T_n(m) = L_n$ (\href{https://oeis.org/A002860}{OEIS~A002860}), evaluating to $L_1=1, \dots, L_6=812{,}851{,}200$.

By restricting to the $n!$ row permutations, the cyclic distributions are highly efficient to compute. Appendix~A provides the complete cyclic data through $n=10$, from which we extract the directed-inventory support sizes $R_n := |\operatorname{supp}(D_n)|$:
\begin{equation*}
    (R_1, \dots, R_{10}) = (1, 1, 2, 5, 16, 59, 234, 979, 4258, 18119).
\end{equation*}

The tables below list $T_n^c(m)$ for the valid range $m=n-1,\dots,(n-1)^2$.
\small
\begin{longtable}{>{\centering\arraybackslash}p{0.06\textwidth} p{0.88\textwidth}}
\toprule
$n$ & $\bigl(T_n^c(n-1),\ldots,T_n^c((n-1)^2)\bigr)$\\
\midrule
1 & $1$\\
2 & $2$\\
3 & $3,0,3$\\
4 & $4,0,4,8,4,0,4$\\
5 & $5,0,5,10,20,20,0,20,20,10,5,0,5$\\
6 & $6,0,6,12,24,48,66,0,66,84,96,84,66,0,66,48,24,12,6,0,6$\\
7 & \ttfamily\seqsplit{7,0,7,14,28,56,112,182,0,182,252,336,420,462,462,0,462,462,420,336,252,182,0,182,112,56,28,14,7,0,7}\\
8 & \ttfamily\seqsplit{8,0,8,16,32,64,128,256,456,0,456,656,928,1280,1696,2096,2416,0,2416,2736,2976,3072,2976,2736,2416,0,2416,2096,1696,1280,928,656,456,0,456,256,128,64,32,16,8,0,8}\\
9 & \ttfamily\seqsplit{9,0,9,18,36,72,144,288,576,1080,0,1080,1584,2304,3312,4680,6444,8514,10719,0,10719,12924,15264,17568,19584,21024,21744,21744,0,21744,21744,21024,19584,17568,15264,12924,10719,0,10719,8514,6444,4680,3312,2304,1584,1080,0,1080,576,288,144,72,36,18,9,0,9}\\
10 & \ttfamily\seqsplit{10,0,10,20,40,80,160,320,640,1280,2470,0,2470,3660,5400,7920,11520,16560,23400,32220,42930,0,42930,53640,66240,80640,96480,113040,129240,143940,156190,0,156190,168440,178240,184640,186880,184640,178240,168440,156190,0,156190,143940,129240,113040,96480,80640,66240,53640,42930,0,42930,32220,23400,16560,11520,7920,5400,3660,2470,0,2470,1280,640,320,160,80,40,20,10,0,10}\\
\bottomrule
\end{longtable}
\normalsize
All cyclic data in this section were generated by exhaustive enumeration of the $n!$ permutations of $\mathbb{Z}_n$, verified by the consistency identities
\[
\sum_m T_n^c(m) = \sum_{\mathbf{m}} D_n(\mathbf{m}) = n!.
\]
For the full Latin-square distribution $T_n(m)$, exact enumeration extends through $n=6$, successfully recovering the known total $L_6 = 812{,}851{,}200$.

\begin{table}[htbp]\centering\small
\caption{All nonzero directed cyclic difference profiles
\(D_{4}(\mathbf m)\).}
\label{tab:D4-profiles}
\renewcommand{\arraystretch}{1.05}
\begin{tabular}{c|c|r|r}
\(m=\Sigma(L_\pi)\) & \(\mathbf m=(M_\pi(1),\ldots,M_\pi(3))\)
& \(D_4(\mathbf m)\) & \(D_n(\mathbf m)/n\)\\ \hline
3 & $(0,0,3)$ & 4 & 1\\
5 & $(0,2,1)$ & 4 & 1\\
6 & $(1,1,1)$ & 8 & 2\\
7 & $(1,2,0)$ & 4 & 1\\
9 & $(3,0,0)$ & 4 & 1\\
\end{tabular}
\end{table}
\begin{table}[htbp]\centering\small
\caption{All nonzero directed cyclic difference profiles
\(D_{5}(\mathbf m)\).}
\label{tab:D5-profiles}
\renewcommand{\arraystretch}{1.05}
\begin{tabular}{c|c|r|r}
\(m=\Sigma(L_\pi)\) & \(\mathbf m=(M_\pi(1),\ldots,M_\pi(4))\)
& \(D_5(\mathbf m)\) & \(D_n(\mathbf m)/n\)\\ \hline
4 & $(0,0,0,4)$ & 5 & 1\\
6 & $(0,0,2,2)$ & 5 & 1\\
7 & $(0,1,1,2)$ & 10 & 2\\
8 & $(0,0,4,0)$ & 5 & 1\\
8 & $(0,2,0,2)$ & 5 & 1\\
8 & $(1,0,1,2)$ & 10 & 2\\
9 & $(0,2,1,1)$ & 10 & 2\\
9 & $(1,0,2,1)$ & 10 & 2\\
11 & $(1,1,2,0)$ & 10 & 2\\
11 & $(1,2,0,1)$ & 10 & 2\\
12 & $(0,4,0,0)$ & 5 & 1\\
12 & $(2,0,2,0)$ & 5 & 1\\
12 & $(2,1,0,1)$ & 10 & 2\\
13 & $(2,1,1,0)$ & 10 & 2\\
14 & $(2,2,0,0)$ & 5 & 1\\
16 & $(4,0,0,0)$ & 5 & 1\\
\end{tabular}
\end{table}
\begin{table}[htbp]\centering
\caption{For \(n=6\), grouping the \(59\) nonzero profiles by the resulting
cyclic total-ascent value \(m\).}
\label{tab:D6-grouped}
\begin{tabular}{c|r|r}
\(m\) & \(\#\{\mathbf m:D_6(\mathbf m)>0,\ \Sigma=m\}\)
& \(T_6^c(m)=\sum_{\Sigma=m}D_6(\mathbf m)\)\\ \hline
5 & 1 & 6\\
7 & 1 & 6\\
8 & 1 & 12\\
9 & 3 & 24\\
10 & 3 & 48\\
11 & 4 & 66\\
13 & 7 & 66\\
14 & 6 & 84\\
15 & 7 & 96\\
16 & 6 & 84\\
17 & 7 & 66\\
19 & 4 & 66\\
20 & 3 & 48\\
21 & 3 & 24\\
22 & 1 & 12\\
23 & 1 & 6\\
25 & 1 & 6\\
\end{tabular}
\end{table}
%

\section{Concluding remarks}
By shifting focus from $\mathcal{L}_n$ to the row-reorderings $L_\pi \in S_n$, cyclic Latin--Eulerian numbers yield an arithmetically rich yet tractable framework. We established exact fiber counts for dominant step profiles via coprimality criteria (Theorem~\ref{thm:two-value}) while quantifying their divergence from the full Latin array. For the total-ascent statistic $T_c$, Theorem~\ref{thm:reduction} reduces the problem to ordinary ascents and boundary endpoints, leading to an explicit inclusion--exclusion formula (Theorem~\ref{thm:E-explicit}).

\end{document}